\documentclass{article}

\usepackage{fullpage}
\usepackage{amsthm}
\usepackage{amsmath}
\usepackage{enumerate}

\usepackage[utf8]{inputenc} 

\usepackage{graphicx}
\usepackage{caption}
\usepackage{subcaption}

\usepackage[justification=centering]{caption}

\newtheorem{theorem}{Theorem}[section]

\newtheorem{corollary}[theorem]{Corollary}

\newtheorem{definition}{Definition}[section]
\newtheorem{example}[theorem]{Example}

\begin{document}
	
	\title{Construction of a surface pencil with a common special surface curve}
	
	\author{Onur Kaya$^{1}$, Mehmet Önder$^{2}$ \\[3mm] {\it $^{1}$ Manisa Celal Bayar University, Department of Mathematics, 45140, Manisa, Turkey} \\ {\it $^{2}$ Independent Researcher, Delibekirli Village, 31440, Kırıkhan, Hatay, Turkey} \date{}}
	
	\maketitle
	
	\begin{abstract}
		In this study, we introduce a new type of surface curves called $D$-type curve. This curve is defined by the property that the unit Darboux vector $\vec{W}_{0} $ of a surface curve $\vec{r}(s)$ and unit surface normal $\vec{n} $ along the curve $\vec{r}(s)$ satisfy the condition $\left\langle \vec{n} ,\vec{W}_{0} \right\rangle =\text{constant}$. We point out that a $D$-type curve is a geodesic curve or an asymptotic curve in some special cases. Then, by using the Frenet vectors and parametric representation of a surface pencil as a linear combination of the Frenet vectors, we investigate necessary and sufficient condition for a curve to be a $D$-type curve on a surface pencil. Moreover, we introduce some corollaries by considering the $D$-type curve as a helix, a Salkowski curve or a planar curve. Finally, we give some examples for the obtained results.
	\end{abstract}
	
	\textbf{AMS Classsification:} 53A04, 53A05, 65D17, 68U0
	

	\section{Introduction}
	\label{intro}
	
	A curve lying on a surface is called a surface curve and a surface can be thought as a family of surface curves. Then, a surface can be constructed by a curve. Of course, that curve may be the set of points of intersection of some different surfaces, i.e., different surfaces can be constructed by the same curve. Especially, some special surface curves lying on a surface are useful tools of differential geometry and computer aided design to study the local geometry of and construction of the surfaces. The most famous surface curves are geodesics, lines of curvature and asymptotic curves and they are used widely in surface design. So, there exist a number of papers which deal with these curves. A geodesic on a surface is the shortest distance between two points of surface on which the curve passing through and a geodesic curve is characterized by vanishing geodesic curvature. Geodesics have many applications in computer vision and image processing, such as in object segmentation \cite{6,18,5} and multi-scale image analysis \cite{17,20}. The concept of geodesic also finds its place in various industrial applications, such as tent manufacturing, cutting and painting path, fiberglass tape windings in pipe manufacturing and textile manufacturing \cite{27,10,3,13,14,11,12,4}.
	
	The other famous special curves on a surface are line of curvature and asymptotic curve. A line of curvature is defined by the property that at all points of the curve the velocity vectors coincide with principal directions of the surface along the curve. So, a line of curvature is a useful tool in surface analysis for exhibiting variations of the principal direction and can be used for the analysis of surfaces, Geometric Design, Shape Recognition, polygonization of surfaces and Surface Rendering \cite{23,1,21}. Moreover, a surface curve is called asymptotic curve if its velocity always points in an asymptotic direction, that is, the direction in which the normal curvature is zero and these curves are used in astronomy, astro-physics and architectural CAD \cite{8,7,9}.
	
	The surface curves given above are also used to construct a surface pencil. Wang and et al have given a method to construct a surface pencil by using common spatial geodesics \cite{26}. Later, Kasap and et al have generalized this method \cite{16}. Then, Li and et al have considered Wang's method for a surface pencil with a common line of curvature \cite{19} and the same method has been considered for a common asymptotic curve by Bayram and et al \cite{2}. 
	
	Moreover, some different special curves can be defined on a surface and a surface pencil can be constructed by these special surface curves. The goal of this paper is to show that a parametric representation of a surface pencil with a common new special surface curve can be given. For this purpose, we define a new surface curve called $D$-type curve and we introduce the necessary and sufficient condition for a given curve to be a $D$-type curve on a parametric surface. We also give a generalization of the marching-scale functions and obtain a sufficient condition for them to satisfy the parametric requirement. Moreover, the obtained results are applied to some special curves such as helix, Salkowski curve and planar curve.
	
	\section{Preliminaries and $D$-type curves}
	\label{sec:1}
	A unit speed curve $\vec{r}(s)$ with arc-length parameter $s$ is said to be regular if its first derivative with respect to $s$ is non-zero, i.e., $\vec{r}'=d\vec{r}/ds\ne \vec{0}$. From now on, all the curves are assumed to be regular in this paper. If the second derivative is zero, then the curve is called a straight line which means that we assume $d^{2} \vec{r}/ds^{2} \ne \vec{0}$ for all $s\in [0,L]$. Such curves are called Frenet curves. Then, on every point of a Frenet curve $\vec{r}(s)$, there is a frame $\{ \vec{T},\vec{N},\vec{B}\} $ which is called the Frenet frame of curve where $\vec{T}=\vec{r}'$, $\vec{N}=\vec{r}''/\left\| \vec{r''}\right\| $ and $\vec{B}=\vec{T}\times \vec{N}$ are the unit tangent, the principal normal and the binormal vectors of the curve, respectively. The Frenet formulae of $\vec{r}(s)$ with respect to $s$ is given by
	\begin{equation} \label{2_1} 
	\left[\begin{array}{c} {\vec{T}'} \\ {\vec{N}'} \\ {\vec{B}'} \end{array}\right]=\left[\begin{array}{ccc} {0} & {\kappa } & {0} \\ {-\kappa } & {0} & {\tau } \\ {0} & {-\tau } & {0} \end{array}\right]\left[\begin{array}{c} {\vec{T}} \\ {\vec{N}} \\ {\vec{B}} \end{array}\right], 
	\end{equation}
	where $\kappa =\kappa (s)$ and $\tau =\tau (s)$ are the curvature and the torsion functions of the curve and they are given by $\kappa =\left\| \vec{r''}\right\| $ and $\tau =\det (\vec{r}',\vec{r}'',\vec{r}''')/\left\| \vec{r}''\right\| ^{2} $, respectively \cite{25}. The vector $\vec{W}_{0} $ defined by 
	\[\vec{W}_{0} =\frac{\tau }{\sqrt{\kappa ^{2} +\tau ^{2} } } \vec{T}+\frac{\kappa }{\sqrt{\kappa ^{2} +\tau ^{2} } } \vec{B},\]
	is called unit Darboux vector of the curve \cite{25}.
	
	A unit speed curve $\vec{\alpha }(s):I\to E^{3} $ with unit tangent $\vec{T}(s)=\vec{\alpha '}(s)$ is called a general helix if there is a constant unit vector $\vec{u}$, so that $\left\langle \vec{T},\vec{u}\right\rangle =\cos \theta $ is constant along the curve, where $\theta \ne \pi /2$. It is well-known that the curve $\vec{\alpha }$ is a general helix if and only if $\frac{\tau }{\kappa } (s)=\text{constant}$ \cite{25}. A curve $\vec{\alpha }$ with $\kappa (s)\ne 0$ is called a slant helix if the principal normal lines of $\vec{\alpha }$ make a constant angle with a fixed direction \cite{15}.
	
	\begin{definition} \label{def21}
		(\cite{24,22}) A Frenet curve $\vec{\alpha }(s):I\to E^{3} $ with constant curvature $\kappa $ and non-constant torsion $\tau $ is called a Salkowski curve. Moreover, a Frenet curve $\vec{\alpha }(s):I\to E^{3} $ with non-constant curvature $\kappa $ and constant torsion $\tau $ is called an anti-Salkowski curve.
	\end{definition}	
	
	The principal normal vectors of a Salkowski curve make a constant angle with a fixed direction \cite{22}. By considering the definition of slant helix, it means that a Salkowski curve is also a slant helix.
	
	A surface is called to be regular if it has a tangent plane at each point. It means that the unit normal vector of a regular surface $\vec{P}:(s,t)\to \vec{P}(s,t)$ at each point is different from zero and defined by the formula
	
	\begin{equation} \label{2_2} 
	\vec{n}(s,t)=\frac{\frac{\partial \vec{P}}{\partial s} \times \frac{\partial \vec{P}}{\partial t} }{\left\| \frac{\partial \vec{P}}{\partial s} \times \frac{\partial \vec{P}}{\partial t} \right\| } . 
	\end{equation}
	
	\begin{definition} \label{def22}
		Let $\vec{P}$ be a regular surface in $E^{3} $ with surface normal $\vec{n}(s,t)$ and $\vec{r}(s)$ be a unit speed isoparametric curve on $\vec{P}$ with Frenet frame $\left\{\vec{T},\vec{N},\vec{B}\right\}$, unit Darboux vector $\vec{W}_{0} $ and let the unit surface normal along the curve $\vec{r}(s)$ be $\vec{n} $. If the unit Darboux vector of the curve $\vec{r}(s)$ and $\vec{n} $ satisfy the condition
		\begin{equation} \label{2_3} 
		\left\langle \vec{n} ,\vec{W}_{0} \right\rangle =\text{constant}, 
		\end{equation} 
		then, the curve $\vec{r}(s)$ is called a $D$-type curve on $\vec{P}$. 
	\end{definition}
	
	Since $\vec{W}_{0}$ lies on the rectifying plane of $\vec{r}(s)$, from Definition \ref{def22} we have that if $\left\langle \vec{n} ,\vec{W}_{0} \right\rangle =0$, then, the surface normal $\vec{n} $ and principal normal $\vec{N}$ are linearly dependent. It means that $\vec{r}(s)$ is a geodesic on $\vec{P}$.  Analogue to that, if $\left\langle \vec{n} ,\vec{W}_{0} \right\rangle =1$, then $\sin \theta = \frac{\sqrt{\kappa^2+\tau^2}}{\kappa}$. So, the surface normal $\vec{n}$ is orthogonal to the principal normal $\vec{N}$, i.e., the curve $\vec{r}(s)$ is an asymptotic curve on $\vec{P}$ if and only if $\left\langle \vec{n} ,\vec{W}_{0} \right\rangle =1$ and $\tau=0$, the curve is a planar curve. These special cases give that the $D$-type curves are more general than geodesics and asymptotic curves. Then, the applications of $D$-type curves contain the applications of geodesics and asymptotic planar curves.
	
	\section{Parametric representation of a surface pencil with a common $D$-type curve}
	\label{sec:2}
	Let $\vec{r}(s)$ be a Frenet curve where $s$ is the arc-length parameter. In order to construct a surface pencil where $\vec{r}(s)$ is the common $D$-type curve, we give the following parameterization of the surface $\vec{P}(s,t):[0,L]\times [0,T]\to E^{3} $ as follows
	\begin{equation} \label{3_1} 
	\vec{P}(s,t)=\vec{r}(s)+u(s,t)\vec{T}(s)+v(s,t)\vec{N}(s)+w(s,t)\vec{B}(s), 
	\end{equation} 
	where $0\le s\le L$, $0\le t\le T$; $u(s,t)$, $v(s,t)$ and $w(s,t)$ are smooth functions and their values indicate, respectively, the extension-like, flexion-like, and retortion-like effects, by the point unit through the time $t$, starting from $\vec{r}(s)$ \cite{26}. In this study, we aim to introduce necessary and sufficient condition for the given curve $\vec{r}(s)$ to be a $D$-type curve on the surface $\vec{P}(s,t)$ where $\vec{P}(s,t)$is assumed as given in the form (\ref{3_1}). 
	
	The unit surface normal along the curve $\vec{r}(s)$ satisfies that $\vec{n} =\cos \theta \vec{N}+\sin \theta \vec{B}$ where the vectors $\vec{N},\, \, \vec{B}$ are principal normal and binormal vectors of $\vec{r}(s)$, respectively. Now, we investigate the necessary and sufficient condition for an isoparametric curve $\vec{r}(s)$ to be a common $D$-type curve on a surface pencil $\vec{P}(s,t)$. Since the curve $\vec{r}(s)$ is an isoparametric curve on $\vec{P}(s,t)$, there is a parameter $t_{0} \in \left[0,T\right]$ such that
	\begin{equation} \label{3_2} 
	\vec{P}(s,t_{0} )=\vec{r}{\kern 1pt} (s),\, \, \, \, \, \, \, \, 0\le t_{0} \le T,\, \, \, \, \, 0\le s\le L. 
	\end{equation} 
	In this paper, we assume that $t_{0} =0$. Then we get $u(s,t_{0} )=v(s,t_{0} )=w(s,t_{0} )=0$. Next, from Definition \ref{def22}, the curve $\vec{r}(s)$ is a $D$-type curve if the surface normal along curve and unit Darboux vector of curve satisfy the condition given by (\ref{2_3}). Therefore, we compute the surface normal by using (\ref{2_2}). In order to do that, using the Frenet formula we obtain the following derivations
	\[\begin{array}{l} {\frac{\partial \vec{P}(s,t)}{\partial s} =\left[1-\kappa (s)v(s,t)+\frac{\partial u(s,t)}{\partial s} \right]\vec{T}(s)} \\ {\qquad \, \, \, \, +\left[\kappa (s)u(s,t)-\tau (s)w(s,t)+\frac{\partial v(s,t)}{\partial s} \right]\vec{N}(s)} \\ {\qquad \, \, \, \, +\left[\tau (s)v(s,t)+\frac{\partial w(s,t)}{\partial s} \right]\vec{B}(s)} \end{array}\] 
	and
	\[\frac{\partial \vec{P}(s,t)}{\partial t} =\frac{\partial u(s,t)}{\partial t} \vec{T}(s)+\frac{\partial v(s,t)}{\partial t} \vec{N}(s)+\frac{\partial w(s,t)}{\partial t} \vec{B}(s).\]
	
	Considering $t=t_0$, we have $u(s,t_{0} )=v(s,t_{0} )=w(s,t_{0} )=0$. By the definition of partial derivative, we obtain
	\begin{equation*}
	\frac{\partial u(s,t_0)}{\partial s}=\frac{\partial v(s,t_0)}{\partial s}=\frac{\partial w(s,t_0)}{\partial s}=0
	\end{equation*}
	which leads us to following equations
	\[\left\{\begin{array}{l} {\frac{\partial \vec{P}(s,t_0)}{\partial s} =\vec{T}(s)}, \\ {\frac{\partial \vec{P}(s,t_0)}{\partial t} =\frac{\partial u(s,t_0)}{\partial t} \vec{T}(s)+\frac{\partial v(s,t_0)}{\partial t} \vec{N}(s)+\frac{\partial w(s,t_0)}{\partial t} \vec{B}(s).} \end{array}\right.\]
	
	Therefore, we have
	\begin{equation} \label{3_3}
	\vec{n}(s,t_0)=\frac{\frac{\partial \vec{P}}{\partial s} \times \frac{\partial \vec{P}}{\partial t} }{\left\| \frac{\partial \vec{P}}{\partial s} \times \frac{\partial \vec{P}}{\partial t} \right\|} =\varphi _{1} (s,t_{0} )\vec{T}(s)+\varphi _{2} (s,t_{0} )\vec{N}(s)+\varphi _{3} (s,t_{0} )\vec{B}(s),
	\end{equation}
	where
	\begin{equation} \label{3_4}
	\varphi _{1} (s,t_{0} )=0,
	\end{equation}
	\begin{equation} \label{3_5}
	\varphi _{2} (s,t_{0} )= -\left\lbrace \left[ \frac{\partial v(s,t_{0} )}{\partial s} \right]^2 + \left[ \frac{\partial w(s,t_{0} )}{\partial s} \right]^2 \right\rbrace^{-1/2}\frac{\partial w(s,t_{0} )}{\partial s},
	\end{equation}
	\begin{equation} \label{3_6}
	\varphi _{3} (s,t_{0} )= \left\lbrace \left[ \frac{\partial v(s,t_{0} )}{\partial s} \right]^2 + \left[ \frac{\partial w(s,t_{0} )}{\partial s} \right]^2 \right\rbrace^{-1/2}\frac{\partial v(s,t_{0} )}{\partial s}.
	\end{equation}
	Since $\vec{n}(s,t_{0} )$, $0\le s\le L$, is unit, we also have
	\begin{equation} \label{3_7}
	\varphi _{1} (s,t_{0} )=0,\, \, \, \, \, \varphi _{2} (s,t_{0} )=\cos \theta ,\, \, \, \, \varphi _{3} (s,t_{0} )=\sin \theta.
	\end{equation} 
	Then, we give the following theorem:
	
	\begin{theorem} \label{thm31}
		A given curve $\vec{r}(s)$ is a $D$-type curve on the surface $\vec{P}(s,t)$ if and only if
		\[\left\{\begin{array}{l} {u(s,t_{0} )=v(s,t_{0} )=w(s,t_{0} )=0,} \\ {\varphi _{1} (s,t_{0} )=0,\, \, \, \, \, \, \varphi _{2} (s,t_{0} )=\pm \sqrt{1-\frac{c^{2} (\kappa ^{2} +\tau ^{2} )}{\kappa ^{2} } } ,} \\ {\varphi _{3} (s,t_{0} )=c\frac{\sqrt{\kappa ^{2} +\tau ^{2} } }{\kappa } .} \end{array}\right. \] 
		satisfy where $0\le t_{0} \le T,\, \, \, \, 0\le s\le L$, $c$ is a real constant, $\kappa ,\, \, \tau $ are the curvature and the torsion functions of the curve $\vec{r}(s)$, respectively.
	\end{theorem}
	
	\begin{proof}
		Let $\vec{r}(s)$ be a $D$-type curve on surface pencil $\vec{P}(s,t)$. From Definition \ref{def22}, we have
		\begin{equation} \label{3_8} 
		\left\langle \vec{n} ,\vec{W}_{0} \right\rangle =c,                        
		\end{equation} 
		where $c$ is a real constant. Since the surface normal is orthogonal to tangent of curve at the same points, from (\ref{3_8}) we obtain
		\[\left\langle \vec{n} ,\vec{B}\right\rangle =c\frac{\sqrt{\kappa ^{2} +\tau ^{2} } }{\kappa } .\] 
		Therefore, we can write
		\[\left\langle \varphi _{2} (s,t_{0} )\vec{N}(s)+\varphi _{3} (s,t_{0} )\vec{B}(s),\, \, \, \vec{B}\right\rangle =c\frac{\sqrt{\kappa ^{2} +\tau ^{2} } }{\kappa } .\] 
		Then we get
		\[\varphi _{3} (s,t_{0} )= \sin \theta =c\frac{\sqrt{\kappa ^{2} +\tau ^{2} } }{\kappa },\] 
		and from (\ref{3_7}), we obtain
		\[\varphi _{2} (s,t_{0} )=\pm \sqrt{1-\frac{c^{2} (\kappa ^{2} +\tau ^{2} )}{\kappa ^{2} } } .\] 
		On the other hand, from (\ref{3_1}), we get $u(s,t_{0} )=v(s,t_{0} )=w(s,t_{0} )=0$ and that finishes the proof of the necessary condition.
		
		Conversely, let
		\[\left\{\begin{array}{l} {u(s,t_{0} )=v(s,t_{0} )=w(s,t_{0} )=0,} \\ {\varphi _{1} (s,t_{0} )=0,\, \, \, \, \, \, \varphi _{2} (s,t_{0} )=\pm \sqrt{1-\frac{c^{2} (\kappa ^{2} +\tau ^{2} )}{\kappa ^{2} } } ,} \\ {\varphi _{3} (s,t_{0} )=c\frac{\sqrt{\kappa ^{2} +\tau ^{2} } }{\kappa } .} \end{array}\right. \] 
		satisfy. Then, the equation (\ref{3_1}) holds and for the surface normal along curve, we have
		\[\vec{n}(s,t_{0} )=\pm \sqrt{1-\frac{c^{2} (\kappa ^{2} +\tau ^{2} )}{\kappa ^{2} } } \vec{N}+c\frac{\sqrt{\kappa ^{2} +\tau ^{2} } }{\kappa } \vec{B},\] 
		and it follows that
		\[\left\langle \vec{n}(s,t_{0} ),\vec{W}_{0} \right\rangle =c=\text{constant}\] 
		Thus, the curve $\vec{r}(s)$ is a $D$-type curve on surface pencil $\vec{P}(s,t)$.
	\end{proof}
	
	From Theorem \ref{thm31}, we have the following corollaries:
	
	\begin{corollary} \label{cor32}
		Let the curve $\vec{r}(s)$ be a $D$-type curve on the surface $\vec{P}(s,t)$. Then, $\vec{r}(s)$ is a Salkowski curve (or a slant helix) with constant curvature $\kappa =a\in {R}$ and non-constant torsion $\tau (s)$ on $\vec{P}(s,t)$ if and only if
		\[\left\{\begin{array}{l} {u(s,t_{0} )=v(s,t_{0} )=w(s,t_{0} )=0,} \\ {\varphi _{1} (s,t_{0} )=0,\, \, \, \, \, \, \varphi _{2} (s,t_{0} )=\pm \sqrt{1-\frac{c^{2} (a^{2} +\tau ^{2} )}{a^{2} } } ,} \\ {\varphi _{3} (s,t_{0} )=c\frac{\sqrt{a^{2} +\tau ^{2} } }{a} .} \end{array}\right. \] 
		hold, where $0\le t_{0} \le T,\, \, \, \, 0\le s\le L$ and $c$ is a real constant.
	\end{corollary}
	
	\begin{corollary} \label{cor33}
		Let the curve $\vec{r}(s)$ be a $D$-type curve on $\vec{P}(s,t)$. Then, $\vec{r}(s)$ is an anti-Salkowski curve with non-constant curvature $\kappa (s)$ and constant torsion $\tau =b\in {R}$ on $\vec{P}(s,t)$ if and only if
		\[\left\{\begin{array}{l} {u(s,t_{0} )=v(s,t_{0} )=w(s,t_{0} )=0,} \\ {\varphi _{1} (s,t_{0} )=0,\, \, \, \, \, \, \varphi _{2} (s,t_{0} )=\pm \sqrt{1-\frac{c^{2} (\kappa ^{2} +b^{2} )}{\kappa ^{2} } } ,} \\ {\varphi _{3} (s,t_{0} )=c\frac{\sqrt{\kappa ^{2} +b^{2} } }{\kappa } .} \end{array}\right. \] 
		hold, where $0\le t_{0} \le T,\, \, \, \, 0\le s\le L$ and $c$ is a real constant.
	\end{corollary}
	
	\begin{corollary} \label{cor34}
		Let the curve $\vec{r}(s)$ be a $D$-type curve on $\vec{P}(s,t)$. Then, $\vec{r}(s)$ is a general helix on $\vec{P}(s,t)$ if and only if
		\[\left\{\begin{array}{l} {u(s,t_{0} )=v(s,t_{0} )=w(s,t_{0} )=0,} \\ {\varphi _{1} (s,t_{0} )=0,\, \, \, \, \, \, \varphi _{2} (s,t_{0} )=\pm \sqrt{1-c^{2} (1+d^{2})} ,} \\ {\varphi _{3} (s,t_{0} )=c\sqrt{1+d^{2} } .} \end{array}\right. \] 
		hold, where $0\le t_{0} \le T,\, \, \, \, 0\le s\le L$, $c$ and $\frac{\tau }{\kappa } (s)=d$ are  real constant.
	\end{corollary}
	
	\begin{corollary} \label{cor35}
		Let the curve $\vec{r}(s)$ be a $D$-type curve on $\vec{P}(s,t)$. Then, $\vec{r}(s)$ is an isogeodesic curve on $\vec{P}(s,t)$ if and only if
		\[\left\{\begin{array}{l} {u(s,t_{0} )=v(s,t_{0} )=w(s,t_{0} )=0,} \\ {\varphi _{1} (s,t_{0} )=0,\, \, \, \, \, \, \varphi _{2} (s,t_{0} )=\pm 1 ,} \\ {\varphi _{3} (s,t_{0} )=0.} \end{array}\right. \] 
		hold, where $0\le t_{0} \le T,\, \, \, \, 0\le s\le L$.
	\end{corollary}
	
	\begin{corollary} \label{cor36}
		Let the curve $\vec{r}(s)$ be a $D$-type curve on $\vec{P}(s,t)$. Then, $\vec{r}(s)$ is both isoparametric and asymptotic planar curve on $\vec{P}(s,t)$ if and only if
		\[\left\{\begin{array}{l} {u(s,t_{0} )=v(s,t_{0} )=w(s,t_{0} )=0,} \\ {\varphi _{1} (s,t_{0} )=0,\, \, \, \, \, \, \varphi _{2} (s,t_{0} )=0 }\\ {\varphi _{3} (s,t_{0} )=1.} \end{array}\right. \] 
		hold, where $0\le t_{0} \le T,\, \, \, \, 0\le s\le L$, and $\kappa ,\tau $ are the curvature and the torsion functions of  $\vec{r}(s)$, respectively.
	\end{corollary}
	
	\begin{corollary} \label{cor37}
		Let the curve $\vec{r}(s)$ be a $D$-type curve on $\vec{P}(s,t)$. Then, $\vec{r}(s)$ is a planar curve on $\vec{P}(s,t)$ if and only if
		\[\left\{\begin{array}{l} {u(s,t_{0} )=v(s,t_{0} )=w(s,t_{0} )=0,} \\ {\varphi _{1} (s,t_{0} )=0,\, \, \, \, \, \, \varphi _{2} (s,t_{0} )=\pm \sqrt{1 -c^{2} } ,} \\ {\varphi _{3} (s,t_{0} )=c.} \end{array}\right. \] 
		where $0\le t_{0} \le T,\, \, \, \, 0\le s\le L$ and $c$ is a real constant.
	\end{corollary}
	
	Now, in order to achieve simpler calculation and analysis, we consider the case that the marching-scale functions $u(s,t),\, \, v(s,t)$ and $w(s,t)$ can be thought as the product of two single valued $C^{1} $-functions. Then we can write,
	\begin{equation} \label{3_9} 
	\left\{\begin{array}{l} {u(s,t)=l(s)U(t),} \\ {v(s,t)=m(s)V(t),\qquad 0\le s\le L,\, \, \, 0\le t\le T.} \\ {w(s,t)=n(s)W(t),} \end{array}\right.  
	\end{equation} 
	where $l(s),\, \, m(s),\, \, n(s)$ are not identically zero. Then, Theorem \ref{thm31} gives following corollary:
	
	\begin{corollary} \label{cor38}
		The curve $\vec{r}(s)$ is a $D$-type curve on the surface pencil $\vec{P}(s,t)$ if and only if
		\begin{equation} \label{3_10} 
		\left\{\begin{array}{l} {U(t_{0} )=V(t_{0} )=W(t_{0} )=0,} \\ {m(s)V'(t_{0} )=c\frac{\sqrt{\kappa ^{2} +\tau ^{2} } }{\kappa } ,} \\ {-n(s)W'(t_{0} )=\pm \sqrt{1-\frac{c^{2} (\kappa ^{2} +\tau ^{2} )}{\kappa ^{2} } } .} \end{array}\right.  
		\end{equation} 
		hold, where $c$ is a real constant, $0\le t_{0} \le T$ and $\kappa ,\, \, \tau $ are the curvature and the torsion functions of the curve $\vec{r}(s)$, respectively.
	\end{corollary}
	
	As a result of the Corollary \ref{cor38}, the conditions obtained in Corollaries \ref{cor32}-\ref{cor37} can be given in the simpler form as given in (\ref{3_10}).
	
	So far, we have dealt with unit speed curves. Of course, if the curve is non-unit speed, then the method can be given by a similar way. For this case, we can give Theorem \ref{thm31} as follows:
	
	\begin{theorem}
		A given non-unit speed curve $\vec{r}(q)$ is a $D$-type curve on the surface $\vec{P}(q,t)$ if and only if
		\[\left\{\begin{array}{l} {u(q,t_{0} )=v(q,t_{0} )=w(q,t_{0} )=0,} \\ {\varphi _{1} (q,t_{0} )=0,\, \, \, \, \, \, \varphi _{2} (q,t_{0} )=\pm \frac{1}{\left\| \vec{r}'\right\| } \sqrt{1-\frac{c^{2} (\kappa ^{2} +\tau ^{2} )}{\kappa ^{2} } } ,} \\ {\varphi _{3} (q,t_{0} )=\frac{c}{\left\| \vec{r}'\right\| } \frac{\sqrt{\kappa ^{2} +\tau ^{2} } }{\kappa } .} \end{array}\right. \] 
		hold, where $0\le t_{0} \le T,\, \, \, \, 0\le q\le L$, $c$ is a real constant, $\kappa ,\, \, \tau $ are the curvature and the torsion functions of the curve $\vec{r}(q)$, respectively.
	\end{theorem}
	
	The corollaries given above for unit speed $D$-type curves can also be given for non-unit speed $D$-type curves as a result of the Theorem \ref{thm31}.
	
	\section{Examples}
	\label{sec:3}
	In this section, we give some examples in order to verify the method. For simpler calculations, we use conditions (\ref{3_10}) and corresponding conditions for non-unit speed curves.
	
	\begin{example}
		Let $\vec{r}(s)$ be a circle which is a planar curve given by the parameterization
		\[\vec{r}(s)=\left(\cos s,\sin s,0\right).\] 
		After a simple computation, one can find that the curvature and the torsion are $\kappa =1,\, \, \, \tau =0$, respectively and Frenet vectors are obtained as follows
		\[\left\{\begin{array}{l} {\vec{T}(s)=\left(-\sin s,\cos s,0\right),} \\ {\vec{N}(s)=\left(-\cos s,-\sin s,0\right),} \\ {\vec{B}(s)=(0,0,1).} \end{array}\right. \] 
		Now, let us choose $l(s)=m(s)=n(s)=1$ and $c=\sqrt{3} /2$. Then the conditions (\ref{3_10}) become
		\[\left\{\begin{array}{l} {U(t)=t,} \\ {V(t)=\frac{\sqrt{3} }{2} t,} \\ {W(t)=\frac{t}{2} .} \end{array}\right. \] 
		Thus, the plot of surface pencil 
		{\scriptsize \[\vec{P}{}_{1} (s,t)=\vec{r}(s)+u(s,t)\vec{T}(s)+v(s,t)\vec{N}(s)+w(s,t)\vec{B}(s), -2\pi \le s\le 2\pi ,\, \, \, 0\le t\le 5, \] }
		is given by Figure \ref{fig1}(a). On the other hand, it is possible to control the shape of the surface pencil by adding some control coefficients to the functions $u,\, \, v,\, \, w$, such as
		\[\left\{\begin{array}{l} {u(t)=xt,} \\ {v(t)=y\frac{\sqrt{3} }{2} t,} \\ {w(t)=z\frac{t}{2} ,} \end{array}\right. \] 
		where $x,\, \, y,\, \, z$ are real constants. So, by taking $x=1/5,\, \, y=1/3,\, \, z=1$, the shape of the surface pencil between the same intervals is given by Figure \ref{fig1}(b).
	\end{example}
	
	\begin{figure}
		\centering
		\begin{subfigure}{.5\textwidth}
			\centering
			\includegraphics[width=2.25in, height=2.25in, keepaspectratio=true]{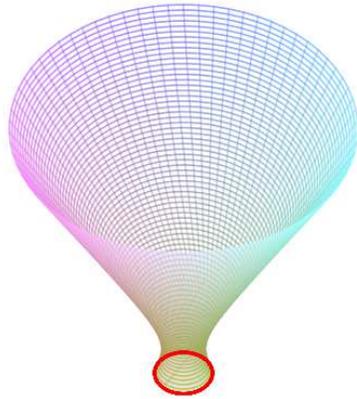}
			\caption{$x=1$, $y=1$, $z=1$}
			\label{fig1a}
		\end{subfigure}%
		\begin{subfigure}{.5\textwidth}
			\centering
			\includegraphics[width=2.25in, height=2.25in, keepaspectratio=true]{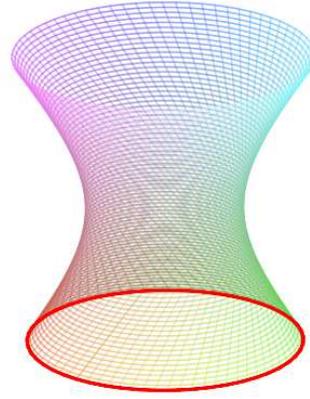}
			\caption{$x=1/5$, $y=1/3$, $z=1$}
			\label{fig1b}
		\end{subfigure}
		\caption{The surface $\vec{P}{}_{1} (s,t)$}
		\label{fig1}
	\end{figure}
	
	\begin{example}
		Let $\vec{r}(s)$ be a cylindrical helix given by the parameterization
		\[\vec{r}(s)=\left(\cos \frac{s}{\sqrt{2} } ,\sin \frac{s}{\sqrt{2} } ,\frac{s}{\sqrt{2} } \right).\] 
		By a simple calculation it is easy to find $\kappa =\tau =1/2$ and
		\[\left\{\begin{array}{l} {\vec{T}(s)=\left(-\frac{1}{\sqrt{2} } \sin \frac{s}{\sqrt{2} } ,\frac{1}{\sqrt{2} } \cos \frac{s}{\sqrt{2} } ,\frac{1}{\sqrt{2} } \right),} \\ {\vec{N}(s)=\left(-\cos \frac{s}{\sqrt{2} } ,-\sin \frac{s}{\sqrt{2} } ,0\right),} \\ {\vec{B}(s)=\left(\frac{1}{\sqrt{2} } \sin \frac{s}{\sqrt{2} } ,-\frac{1}{\sqrt{2} } \cos \frac{s}{\sqrt{2} } ,\frac{1}{\sqrt{2} } \right).} \end{array}\right. \] 
		Taking $l(s)=m(s)=n(s)=1,\, \, \, c=1/2$ and 
		\[\left\{\begin{array}{l} {U(t)=t,} \\ {V(t)=\frac{\sqrt{2}}{2}t,} \\ {W(t)=\frac{\sqrt{2}}{2}t,} \end{array}\right. \] 
		gives us the shape of the surface pencil
		{\scriptsize	\[\vec{P}_{2} (s,t)=\vec{r}(s)+u(s,t)\vec{T}(s)+v(s,t)\vec{N}(s)+w(s,t)\vec{B}(s), -2\pi \le s\le 2\pi ,\, \, \, 0\le t\le 1/4\] }
		as in Figure \ref{fig2}(a).	Considering the control coefficients $x=1,\, \, y=15,\, \, z=5$, we obtain the shape as in Figure \ref{fig2}(b) for the same intervals of $s$ and $t$.
	\end{example}
	
	\begin{figure}
		\centering
		\begin{subfigure}{.5\textwidth}
			\centering
			\includegraphics[width=2.25in, height=2.25in, keepaspectratio=true]{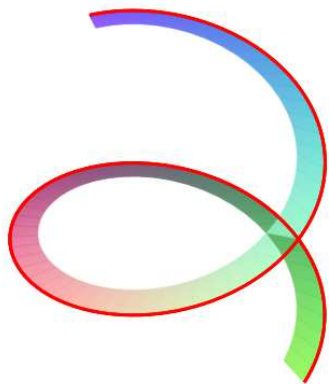}
			\caption{$x=1$, $y=1$, $z=1$}
			\label{fig2a}
		\end{subfigure}%
		\begin{subfigure}{.5\textwidth}
			\centering
			\includegraphics[width=2.25in, height=2.25in, keepaspectratio=true]{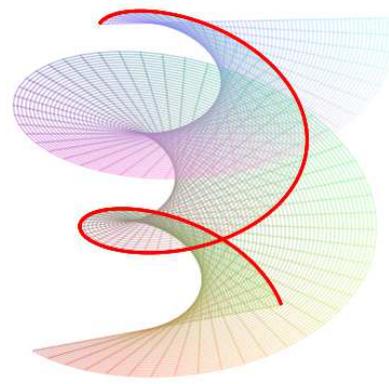}
			\caption{$x=1$, $y=15$, $z=5$}
			\label{fig2b}
		\end{subfigure}
		\caption{The surface $\vec{P}{}_{2} (s,t)$}
		\label{fig2}
	\end{figure}
	
	\begin{example}
		Let \textit{$\vec{r}(q)$} be the eight curve given by the parameterization
		\[\vec{r}(q)=\left(\sin q,\, \, \sin q\cos q,\, \, 0\right).\] 
		Notice that the curve \textit{$\vec{r}(q)$} is not a unit speed curve. Then, we use Theorem \ref{thm31} for the computation. After computing the Frenet elements, we choose $l(q)=1,$ $m(q)=1,$ $n(q)=1$ and $c=\sqrt{3} /2$. Thus, we get
		\[\left\{\begin{array}{l} {U(t)=t,} \\ {V(t)=\frac{\sqrt{3}}{2} \frac{t}{\sqrt{4\cos^4{q}-3\cos^{2}{q}+1}},} \\ {W(t)=\frac{1}{2} \frac{t}{\sqrt{4\cos^4{q}-3\cos^{2}{q}+1}}.} \end{array}\right. \] 
		Therefore, the shape of the surface
		{\scriptsize	\[\vec{P}_{3} (q,t)=\vec{r}(q)+u(q,t)\vec{T}(q)+v(q,t)\vec{N}(q)+w(q,t)\vec{B}(q), 0\le s\le 2\pi ,\, \, \, 0\le t\le 1\] }
		is given as in Figure \ref{fig3}(a). By taking the control coefficients $x=10,\, \, y=1/5,\, \, z=1$, Figure \ref{fig3}(a) becomes Figure \ref{fig3}(b).
	\end{example}
	
	\begin{figure}
		\centering
		\begin{subfigure}{.5\textwidth}
			\centering
			\includegraphics[width=2.25in, height=2.25in, keepaspectratio=true]{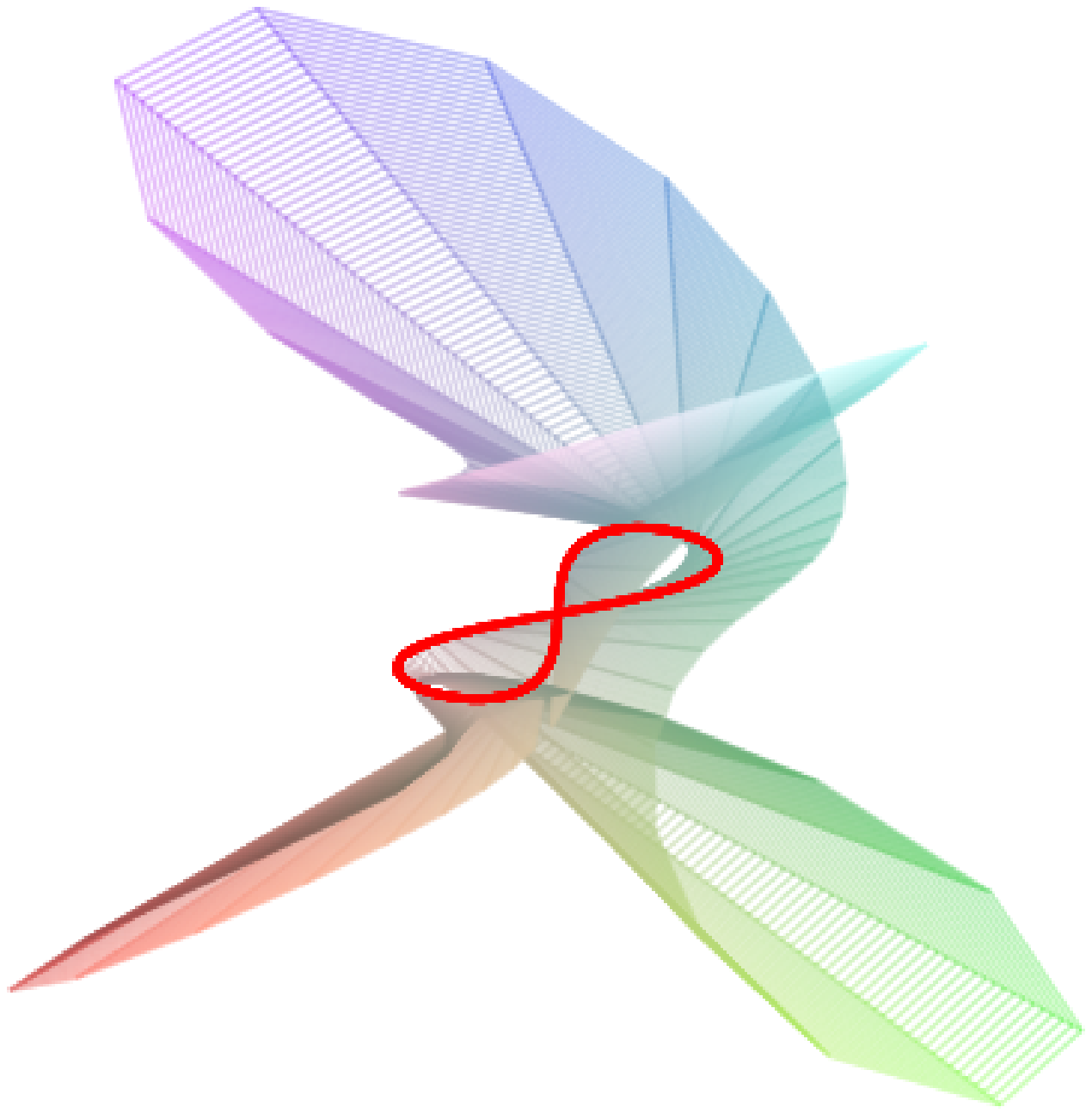}
			\caption{$x=1$, $y=1$, $z=1$}
			\label{fig3a}
		\end{subfigure}%
		\begin{subfigure}{.5\textwidth}
			\centering
			\includegraphics[width=2.25in, height=2.25in, keepaspectratio=true]{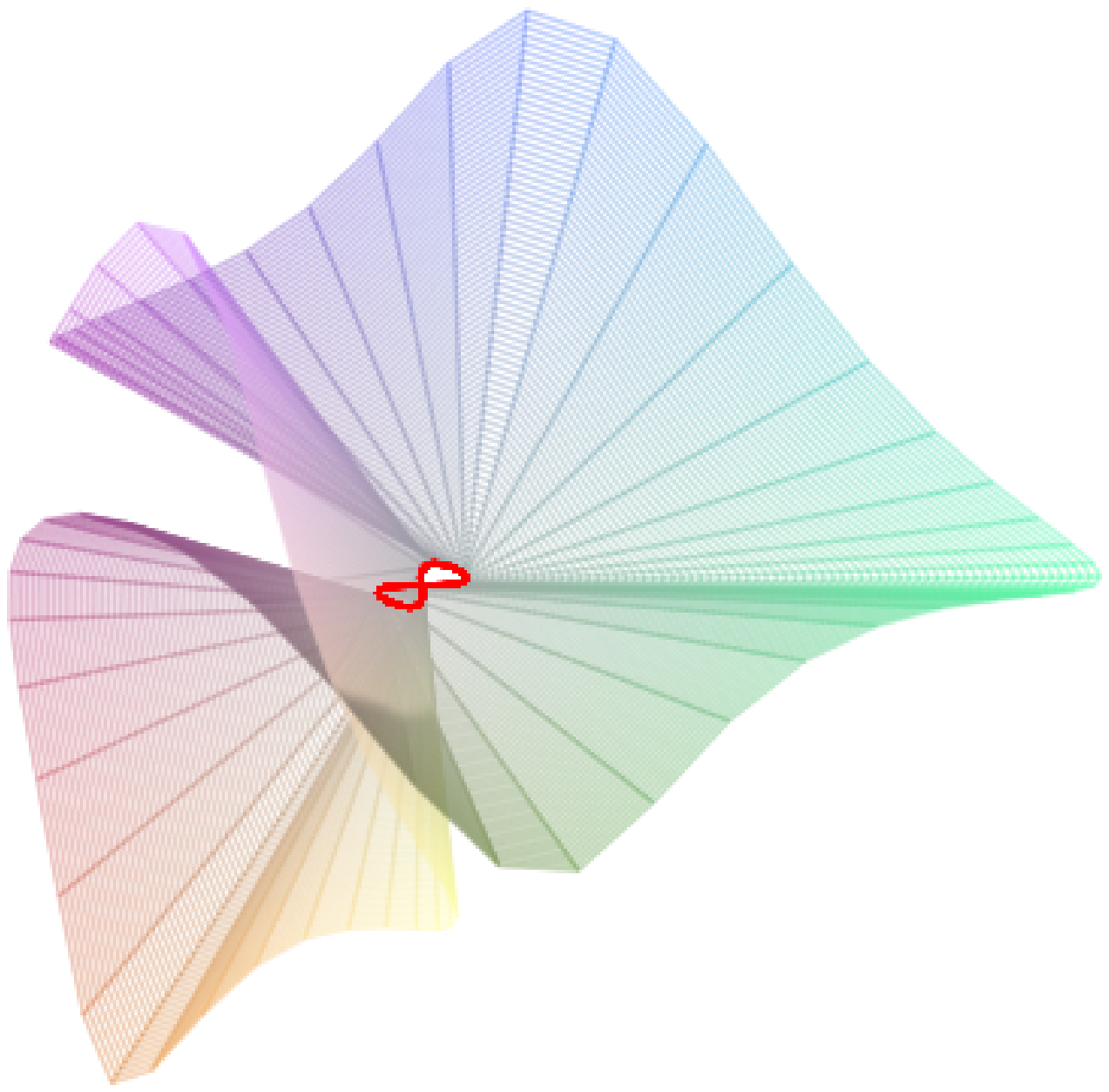}
			\caption{$x=10$, $y=1/5$, $z=1$}
			\label{fig3b}
		\end{subfigure}
		\caption{The surface $\vec{P}{}_{3} (s,t)$}
		\label{fig3}
	\end{figure}

	\begin{example}
		Let the curve $\vec{r}(q)$ be a Salkowski curve given by the parameterization
		{\scriptsize \[\vec{r}(q)=\left(\frac{5}{\sqrt{26} } \left(\frac{\sqrt{26} -26}{104+8\sqrt{26} } \sin \left(\left(1+\frac{\sqrt{26} }{13} \right)q\right)+\frac{\sqrt{26} +26}{-104+8\sqrt{26} } \sin \left(\left(1-\frac{\sqrt{26} }{13} \right)q\right)-\frac{1}{2} \sin q\right),\right. \] \[\begin{array}{l} {\frac{5}{\sqrt{26} } \left(\frac{26-\sqrt{26} }{104+8\sqrt{26} } \cos \left(\left(1+\frac{\sqrt{26} }{13} \right)q\right)-\frac{\sqrt{26} +26}{-104+8\sqrt{26} } \cos \left(\left(1-\frac{\sqrt{26} }{13} \right)q\right)-\frac{1}{2} \cos q\right),} \\ {\left. \frac{25}{4\sqrt{26} } \cos \left(\frac{\sqrt{26} }{13} q\right)\right).} \end{array}\]}
		Since the curve $\vec{r}(q)$ is not a unit speed curve, we use Theorem \ref{thm31} for the computation. After computing the Frenet elements, we choose $l(q)=1,$ $m(q)=1,$$n(q)=1$ and $c=\sqrt{3} /2$. Thus, we get
		\[\left\{\begin{array}{l} {U(t)=t,\, \, \, \, \, V(t)=\frac{\sqrt{78} }{10\cos ^{2} \left(\frac{\sqrt{26} }{26} q\right)} t,} \\ {W(t)=\frac{\sqrt{26}}{10} \frac{\sqrt{1-3\tan^{2} \left({\frac{\sqrt{26}}{26}q}\right)}}{\cos \left( {\frac{\sqrt{26} }{26}q} \right)} t.} \end{array}\right. \] 
		Therefore, the shape of the surface
		{\scriptsize	\[\vec{P}_{4} (q,t)=\vec{r}(q)+u(q,t)\vec{T}(q)+v(q,t)\vec{N}(q)+w(q,t)\vec{B}(q), 0\le q\le 2\pi ,\, \, \, 0\le t\le 1\] }
		is given as in Figure \ref{fig4}(a). By taking the control coefficients $x=-1/10,\, \, y=-1/10,\, \, z=1$, we obtain Figure \ref{fig4}(b) for the same interval of $q$ and $-4\le t\le 4$.
	\end{example}
	
	\begin{figure}
		\centering
		\begin{subfigure}{.5\textwidth}
			\centering
			\includegraphics[width=2.25in, height=2.25in, keepaspectratio=true]{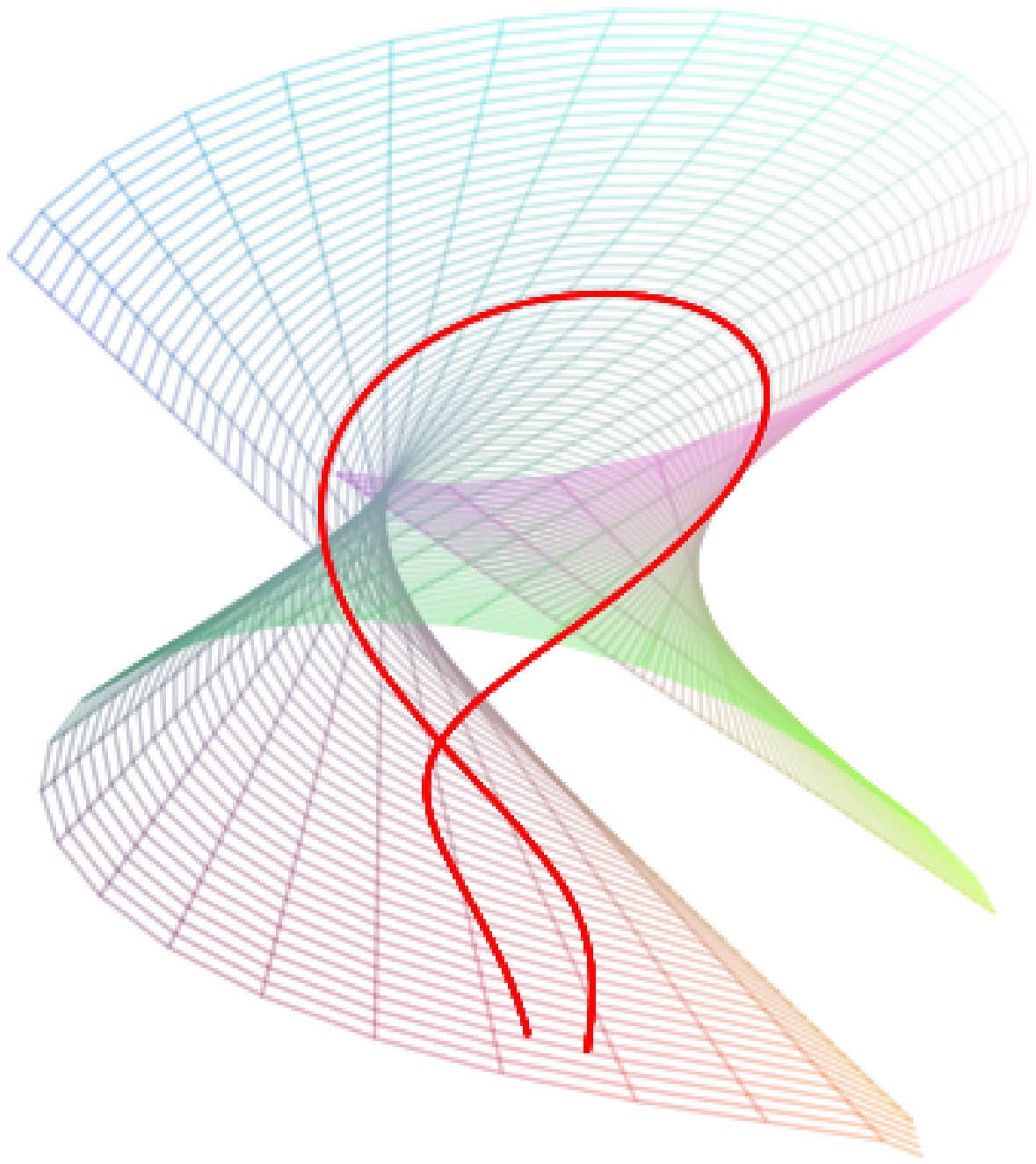}
			\caption{$x=1$, $y=1$, $z=1$}
			\label{fig4a}
		\end{subfigure}%
		\begin{subfigure}{.5\textwidth}
			\centering
			\includegraphics[width=2.25in, height=2.25in, keepaspectratio=true]{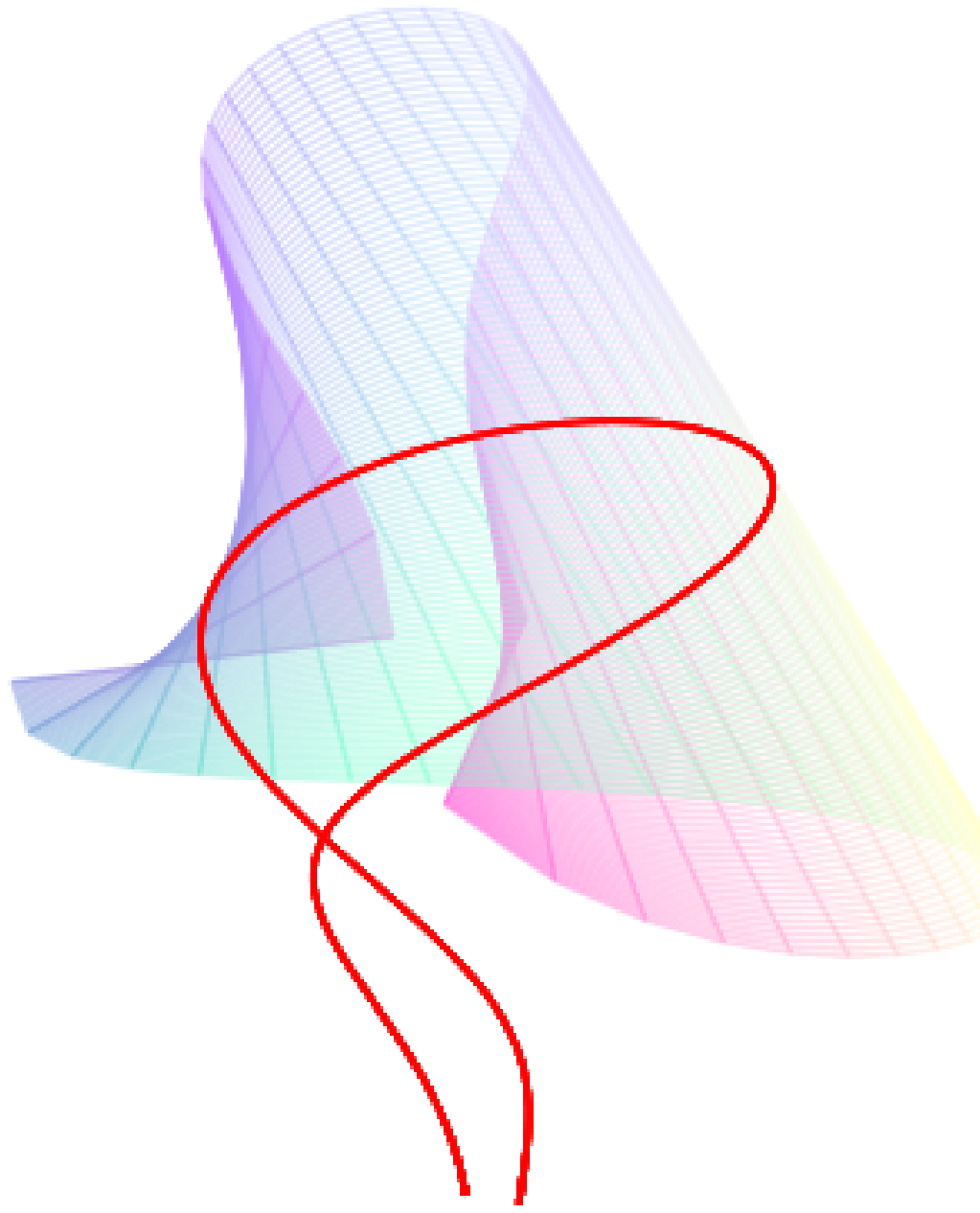}
			\caption{$x=-1/10$, $y=-1/10$, $z=-1$}
			\label{fig4b}
		\end{subfigure}
		\caption{The surface $\vec{P}{}_{4} (s,t)$}
		\label{fig4}
	\end{figure}
	
	\section{Conclusions}
	\label{sec:4}
	A new type of surface curves is defined by the property that the unit Darboux vector $\vec{W}_{0} $ of the curve and unit surface normal $\vec{n} $ along the curve satisfies the condition $\left\langle \vec{n} ,\vec{W}_{0} \right\rangle =\text{constant}$. Such curves are called $D$-type curves and it is pointed out that a $D$-type curve can be a geodesic or an asymptotic curve in some special cases. It means that $D$-type curves is a larger class of surface curves and they are used as a geometric tool to introduce a surface pencil with common $D$-type curve. The necessary and sufficient condition for a curve to be a common $D$-type curve on a surface pencil is obtained. Then, the obtained results are considered for some special cases such as the common $D$-type curve is a helix, Salkowski curve or a planar curve. By using different types of marching-scale functions, a surface pencil with a common $D$-type curve is given parametrically. 
	
	Of course, the local differential geometry and other applications of $D$-type curves to CAD are open problems. For instance, under the same condition, the curve can be thought non-isoparametric on the parametric surface. Moreover, it possible to consider $D$-type curves to construct surfaces given by implicit form $F(x,y,z)=0$.

\end{document}